\documentclass{article}
\usepackage{amsrefs}
\usepackage{amssymb}
\usepackage{amsmath}
\usepackage{amsthm}
\usepackage{graphicx}

\newtheorem*{theorem}{Theorem}
\newtheorem{lemma}{Lemma} 
\DeclareMathOperator{\val}{val} \DeclareMathOperator{\Supp}{Supp} \DeclareMathOperator{\divisor}{div}
\newtheorem*{subject}{2000 Mathematics Subject Classification}
\newtheorem*{keywords}{Keywords}

\author{Marc Coppens\footnote{Thomas Moore Campus Geel, Departement IBW,
Kleinhoefstraat 4, B-2440 Geel, Belgium; K.U.Leuven, Departement of Mathematics,
Celestijnenlaan 200B, B-3001 Leuven, Belgium; email: marc.coppens@khk.be.
Partially supported by the FWO-grant 1.5.012.13N}}
\title{Clifford's Theorem for graphs}
\date{}

\begin{document}
\maketitle \noindent

\begin{abstract}
Let $\Gamma$  be a metric graph of genus $g$. Assume there exists a natural number $2 \leq r \leq g-2$ such that $\Gamma$ has a linear system $g^r_{2r}$. Then $\Gamma$ has a linear system $g^1_2$. For algebraic curves this is part of the well-known Clifford's Theorem.
\end{abstract}

\begin{subject}
14T05; 05C99
\end{subject}

\begin{keywords}
graph, divisor, linear system, Clifford's Theorem
\end{keywords}

\section{Introduction}\label{section1}

A well-known and important theorem in the theory of linear systems on a smooth projective irreducible algebraic curve $C$ of genus at least 2 is Clifford's Theorem.
Let $g^r_d$ be a linear system on $C$ then $d\geq 2r$. Moreover in case $d=2r$ then either $g^r_d$ is trivial or it is the canonical linear system or the curve is hyperelliptic and $g^r_{2r}=rg^1_2$.
For a proof of Clifford's Theorem see e.g. \cite{ref1}*{Chapter III}.
The inequality $d\geq 2r$ is an easy corollary of the Riemann-Roch Theorem and also the characterisation of special linear systems $g^r_{2r}$ on a hyperelliptic curve can be obtained directly from the Riemann-Roch Theorem.

In \cite{ref2} one proves a Riemann-Roch Theorem for finite graphs and in \cite{ref3}, \cite{ref4} one proves a Riemann-Roch Theorem for metric graphs and for tropical curves. In particular the inequality $d\geq 2r$ for a special $g^r_d$ on such structure follows immediately.
In \cite{ref5} one gives a description of the special linear systems $g^r_{2r}$ on such structure in case there exist a $g^1_2$.
Now we are going to prove that, in case such structure has a $g^r_{2r}$ for some $2\leq r\leq g-2$, then there exists a $g^1_2$, thereby obtaining  a complete Clifford's Theorem for such structures.

In order to mention the main steps of the proof we give a proof of a part of Clifford's Theorem for the classical case of curves that will be the prototype for the proof for metric graphs.
So assume C is a smooth curve of genus $g$ possessing a $g^r_{2r}$ with $2\leq r\leq g-2$. We are going to prove that $C$ has a $g^{r-1}_{2r-2}$.
Take $D_1\in C^{(r)}$ general and take $E_1\in C^{(r)}$ such that $D_1+E_1\in g^r_{2r}$.
From the Riemann-Roch Theorem we obtain $\mid K_C-g^r_{2r}\mid = g^{g-1-r}_{2g-2-2r}$.
Take $D_2\in C^{(g-1-r)}$ general and take $E_2\in C^{(g-1-r)}$ such that $D_2+E_2\in g^{g-1-r}_{2g-2r-2}$.
Let $D=D_1+D_2$ and let $E=E_1+E_2$, then $D\in C^{(g-1)}$ is general, hence $\dim \mid D\mid = \dim \mid K_C-D \mid = 0$, and $D+E$ is canonical, hence $\mid K_C-D \mid = \{ E \}$.
Since $\dim \mid D_2 \mid =0$ one has $\dim \mid g^r_{2r}+E_2 \mid = r$ hence $E_2$ is the fixed divisor for $\mid g^r_{2r} + E_2 \mid$, hence for $P\in D_1$ it is contained in the fixed divisor of $\mid \left( g^r_{2r} -P \right) + E_2 \mid $.

Let $P\in D_1$ and let $P'\in D_2$. Using $D'_1=D_1-P+P'\in C^{(r)}$ there exists $E'_1 \in C^{(r)}$ with $D'_1+E'_1 \in g^r_{2r}$.
Similarly, using $D'_2=D_2-P'+P\in C^{(g-1-r)}$ there exists $E'_2\in C^{(g-1-r)}$ such that $D'_2+E'_2\in \mid K_C-g^r_{2r}\mid$.
Hence $D'_1+D'_2+E'_1+E'_2$ is canonical but $D'_1+D'_2=D$ hence $E'_1+E'_2=E$.
Assume $E'_2=E_2$ then $E'_1=E_1$. Let $D'=D_1\cap D_2$ then $\deg (D')=r-1$, $D_1=D'+P$, $D'_1=D'+P'$ hence $D'+P+E_1\in g^r_{2r}$ and $D'+P'+E_1\in g^r_{2r}$. It follows $P=P'$ but this is not possible because $D_1$ and $D_2$ are general divisors.
However $D'+E=(D_1-P)+E_1+E_2\in g^r_{2r}(-P)+E_2$ and $D'+E=(D'_1-P')+E'_1+E'_2\in g^r_{2r}(-P')+E'_2$, hence $\mid g^r_{2r}(-P)+E_2 \mid = \mid g^r_ {2r}(-P')+E'_2 \mid$.
Since $E_2$ is contained in the fixed divisor of $\mid g^r_{2r}(-P)+E_2 \mid$ and $E_2 \neq E'_2$ the linear system $g^r_{2r}(-P)$ has a fixed point.
This proves $C$ has a $g^{r-1}_{2r-2}$.

The advantage of this proof for adaptation to the case of metric graphs is that it makes no explicit use of linear algebra as in the proof given in e.g. \cite{ref1}*{Chapter III}.
In the case of metric graphs it is not clear how to adapt those linear algebra techniques in a suitable way. A remaining difficulty is the use of fixed divisors in the above given proof. In the case of curves, if $F$ is an effective divisor such that $\dim \mid D+F \mid = \dim \mid D \mid$ then for all $E \in \mid D+F \mid$ one has $E \geq F$. For metric graphs this is not true any more.

In the first part of the proof, for a metric graph $\Gamma$ we describe a situation that is very similar to the situation of curves. In that situation we obtain corresponding effective divisors $D$ and $E$ of degree $g-1$ together with decompositions $D=D_1+D_2$ and $E=E_1+E_2$ using $g^r_{2r}$ and $ \mid K_{\Gamma}-g^r_{2r} \mid$ such that the linear systems associated to those divisors do not contain other divisors. This description and the rest of the proof makes intensive use of the Abel-Jacobi maps as defined in \cite{ref4}.
In the second part of the proof, within this situation we can repeat the first steps of the proof for the case of curves. Then  in order to avoid the use of the fixed divisor of a linear system, we start making more complicated arguments.
From those aguments, using a rank determining set described in \cite{ref6}, one obtains the remaining part of the Clifford's Theorem. So clearly the proof is not a direct corollary of using the Riemann-Roch Theorem.

The case of tropical curves is an immediate corrolary of the case of metric graphs.
A tropical curve $C$ can be defined as being a metric graph $\Gamma$ together with some so-called unbounded edges (see e.g. \cite{ref3}). It is well-known that there is a bijective correspondence between linear systems on $C$ and on $\Gamma$.
Also the case of finite graphs follows from the case of metric graphs.
Associated to a finite graph $G$ there is a metric graph $\Gamma$ using an interval of length one to realize each edge.
A divisor $D$ on $G$ can also be considered as a divisor on $\Gamma$ and it is proved in \cite{ref7} that the ranks of $D$ on $G$ and on $\Gamma$ are equal.
Hence if $G$ has a $g^r_{2r}$ then $\Gamma$ has a $g^r_{2r}$. In that case our theorem implies $\Gamma$ has a $g^1_2$.
Metric graphs having a $g^1_2$ can be described by means of a harmonic morphism of degree two from $\Gamma$ to a tree $T$ (with one exception, in which case $G$ is trivially hyperelliptic) (see \cite{ref8}).
Using this description one can prove the existence of a divisor $D$ in $G$ belonging to $g^1_2$ as a divisor on $\Gamma$.
Then again from \cite{ref7} one obtains $D$ defines a $g^1_2$ on $G$.

To summarise, in this paper we prove the following theorem.

\begin{theorem}
Let $\Gamma$ be a metric graph of genus $g\geq 4$ and let $r$ be an integer satisfying $2\leq r\leq g-2$ such that $\Gamma$ has a linear system $g^r_{2r}$ then $\Gamma$ has a linear system $g^1_2$.
\end{theorem}

\section{Some generalities}\label{section2}

A \emph{metric graph} $\Gamma$ is a compact connected metric space such that for each point $P$ on $\Gamma$ there exists $n \in \mathbb{Z}$ with $n \geq 1$ such that some neighbourhood of $P$ is isometric to $\{ z \in \mathbb{C} : z=te^{2k\pi i/n} \text{ with } t \in [0,r] \subset \mathbb{R} \text{ and } k\in \mathbb{Z} \}$ for some $r>0$ and with $P$ corresponding to $0\in \mathbb{C}$.
This number $n$, depending on $P$, is called the \emph{valence} $\val (P)$ of $P$.
For a metric graph $\Gamma$ we define the \emph{set of vertices} by
\[
V(\Gamma) = \{ P\in \Gamma : \val (P)\neq 2 \} \text { .}
\]
The \emph{set of edges} $E(\Gamma)$ is the set of connected components of $\Gamma \setminus V(\Gamma)$.
For such a component $e$ (which is open in $\Gamma$) there is a homeomorphism to some finite open interval of $\mathbb{R}$ that is locally isometric.
Fixing such an homeomorphism for each edge $e$ we have an identification $e \subset \mathbb{R}$ as being a finite open interval.
The set $\bar{e} \setminus e$ (here $\bar{e}$ is the closure of $e$ in $\Gamma$) is a subset of $V(\Gamma)$ containing at most two points.
In case it consists of one point the $e$ is called a \emph{loop}.
The \emph{genus} of the graph $\Gamma$ is given by the formula
\[
g(\Gamma)= \mid E(\Gamma) \mid - \mid V(\Gamma) \mid + 1 \text { .}
\]

A \emph{divisor} on a metric graph $\Gamma$ is a finite formal integer linear combination of points on $\Gamma$.
In case $D$ is a divisor and $P$ is a point on $\Gamma$ then we write $D(P)$ to denote the coefficient of $D$ at $P$, hence 
\[
D=\sum_{P \in \Gamma} D(P)P \text { .}
\]
The sum of the coefficients of a divisor $D$ is called the \emph{degree} of the divisor
\[
\deg (D)=\sum _{P \in \Gamma} D(P) \text{ .}
\]
The \emph{support} of $D$ is defined by
\[
\Supp (D)= \{ P \in \Gamma : D(P) \neq 0 \} \text { .}
\]
A divisor $D$ is called \emph{effective} if $D(P)\geq 0$ for all $P\in \Gamma$ and we write $D\geq 0$ in that case.
More generally we write $D_1 \geq D_2$ in case $D_1 - D_2 \geq 0$.
The \emph{canonical divisor} is defined by
\[
K_{\Gamma}=\sum_{P\in \Gamma} (\val (P)-2)P \text { .}
\]

A \emph{principal divisor} on $\Gamma$ is defined by a \emph{rational function} $f$ on $\Gamma$.
This is a continuous function $f:\Gamma \rightarrow \mathbb{R}$ such that for each $e\in E(\Gamma)$ there exists a finite partition $e=e_1 \cup \cdots \cup e_n$ of $e\subset \mathbb{R}$ in subintervals such that $f$ restricted to $e_i$ is affine with integer slope.
For $P \in \Gamma $ we define $\divisor (f)(P)$ as being the sum of all slopes of $f$ on $\Gamma$ in all directions emanating from $P$.
Two divisors $D_1$ and $D_2$ on $\Gamma$ are \emph{linearly equivalent} in case $D_2-D_1$ is a principal divisor and we write $D_1 \sim D_2$ in that case. Linearly equivalent divisors have the same degree.
For a divisor $D$ on $\Gamma$ the \emph{linear system} $\mid D \mid$ is the set of all effective divisors linearly equivalent to $D$.
The \emph{rank} $r(D)$ of a divisor $D$ on $\Gamma$ is -1 in case $\mid D \mid$ is empty. In case $\mid D \mid$ is not empty then it is the maximal integer $r$ such that for each effective divisor $E$ of degree $r$ one has $\mid D-E \mid \neq \emptyset$.
A linear system $\mid D \mid$ with $r(D)=r\geq 0$ and $\deg (D)=d$ is called a linear system $g^r_d$ on $\Gamma$.
There is a Riemann-Roch Theorem
\[
r(D)=\deg (D) - g(\Gamma) +1 +r(K_{\Gamma}-D) \text { .}
\]

A very important tool in the study of divisors on metric graphs is the use of reduced divisors.
We are going to use the following characterisation coming from \cite{ref6}*{Lemma 2.4}.
Let $D$ be a divisor on $\Gamma$ and let $P\in \Gamma$. For each $S\subset \Supp (D) \setminus \{ P \}$ let $\mathcal{U}_{S,P}$ be the connected component of $\Gamma \setminus S$ containing P and let $\mathcal{U}^c_{S,P}$ be the complement $\Gamma \setminus \mathcal{U}_{S,P}$.
The divisor $D$ is \emph{$P$-reduced} if for each $S \subset \Supp (D) \setminus \{ P \}$ the closed set $\mathcal{U}^c_{S,P}$ has a non-saturated boundary point $Q$ with respect to $D$.
A boundary point $Q$ is called \emph{saturated} if $D(Q)$ is at least equal to the outgoing degree of $\mathcal{U}^c_{S,P}$ at $Q$.
Here the \emph{outgoing degree} is the number of segments leaving $\mathcal{U}^c_{S,P}$ at $Q$.
It is well-known that for each divisor $E$ on $\Gamma$ and for each point $P$ on $\Gamma$ there is a unique $P$-reduced divisor $D$ on $\Gamma$ linearly equivalent to $E$.
In case $\mid E \mid\neq \emptyset$ then this divisor $D$ is effective (hence belongs to $\mid E\mid$) and for each $D' \in \mid E \mid$ one has $D(P) \geq D(P')$. This gives rise to the following easy but useful lemma.
\begin{lemma}\label{lemma1}
Let $D$ be an effective divisor on $\Gamma$ and assume for all $P \in \Gamma$ the divisor $D$ is $P$-reduced. Then $\mid D \mid = \{ D \}$.
\end{lemma}
\begin{proof}
For $P \notin \Supp (D)$, since $D$ is $P$-reduced there is no $D'\in \mid D \mid$ with $D'(P) \geq 1$. Hence if $D' \in \mid D \mid$ then $\Supp (D') \subset \Supp (D)$. But for $P \in \Supp (D)$ also $D'(P) \leq D(P)$. This proves $D'=D$.
\end{proof}
\begin{lemma}\label{lemma2}
Let $D$ be an effective divisor on $\Gamma$ and assume there exists $P\in \Gamma$ such that $D$ is not $P$-reduced. There exists $D' \in \mid D \mid$ and $Q \in V(\Gamma)$ such that $Q \in \Supp (D')$.
\end{lemma}
\begin{proof}
Of course we can assume that for all $Q \in V(\Gamma)$ one has $Q \notin \Supp (D)$.
Since $D$ is not $P$-reduced there exists $S \subset \Supp(D) \setminus \{ P \}$ such that each boundary point of $\mathcal{U}^c_{S,P}$ is saturated.
\begin{figure}[h]
\begin{center}
\includegraphics[height=2 cm]{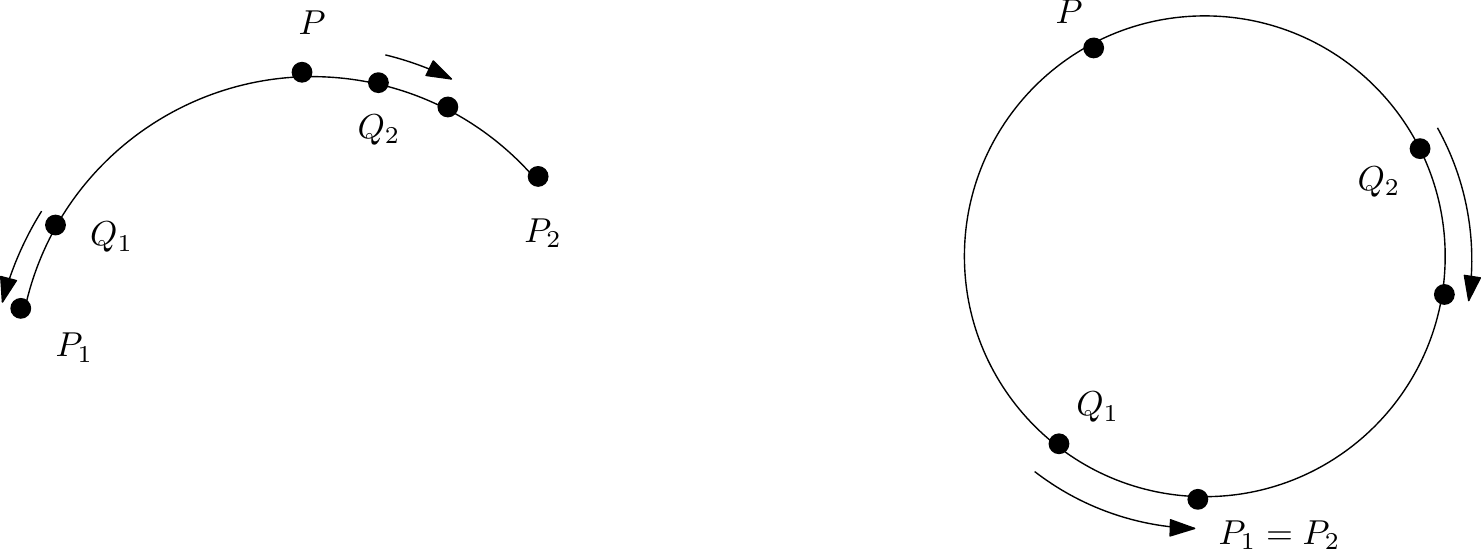}
\caption{$\mathcal{U}_{S,P}$ contains no vertex }\label{Figuur 1}
\end{center}
\end{figure}

First assume $\mathcal{U}_{S,P}$ does not contain a point of $V(\Gamma)$. There exists $e \in E(\Gamma)$ such that $\mathcal{U}_{S,P} \subset e$.
The boundary points of $\mathcal{U}^c_{S,P}$ (contained in $S$, hence not contained in $V(\Gamma)$) are inside $e$.
Let $\bar{e}\setminus e=\{ P_1, P_2 \}$ ($P_1=P_2$ is possible).
Since $P\in \mathcal{U}_{S,P}$ there is a point of $S$ on both connected components of $e \setminus \{ P \}$.
Let $Q_1$ and $Q_2$ be the points on those components most close to $P_1$ and $P_2$ (of course $Q_1 \neq Q_2$).
Then $Q_1+Q_2$ is linearly equivalent to an effective divisor of degree 2 with support contained in $\bar{e}$ and having at least one of the points $P_1$ and $P_2$ in its support (see Figure \ref{Figuur 1}).
\begin{figure}[h]
\begin{center}
\includegraphics[height=2 cm]{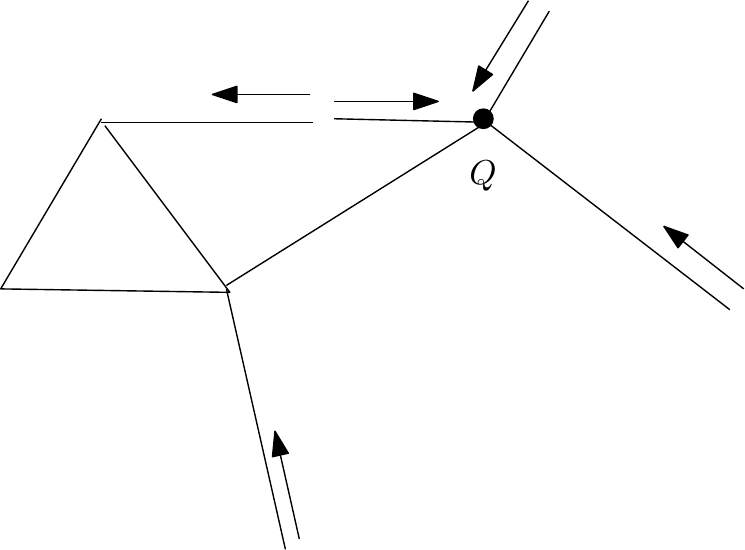}
\caption{$\mathcal{U}_{S,P}$ contains a vertex }\label{Figuur 2}
\end{center}
\end{figure}

Next, assume $\mathcal{U}_{S,P} \cap V(\Gamma)\neq \emptyset$ and let $Q$ be point in $V(\Gamma)$ most close to a boundary point of $\mathcal{U}^c_{S,P}$.
Then there is a subdivisor $D'$ of $D$ (with support equal to $S$) such that $D'$ is linearly equivalent to an effective divisor $D''$ with support contained in $\mathcal{U}_{S,P}$ and containing $Q$ in its support (see Figure \ref{Figuur 2}).

\end{proof}

\section{First part of the proof}\label{section3}

In this part of the proof we are going to describe the situation in which we can perform the first steps of the proof given for curves. In order to obtain this situation we are going to use the Abel-Jacobi maps desribed in \cite{ref4}. First we recall some details on the structure of those Abel-Jacobi maps.

In \cite{ref4}*{6.1} one defines the \emph{tropical Jacobian} $J(\Gamma)$ of a metric graph $\Gamma$. The authors do introduce one-forms on $\Gamma$ and $\Omega (\Gamma)$ is the \emph{space of global one-forms} on $\Gamma$. It is a $g$-dimensional $\mathbb{R}$-vectorspace. For $e \in E(\Gamma)$ identified with an open interval $e \subset \mathbb{R}$ the restriction of $\omega \in \Omega (\Gamma)$ to $e$ corresponds to a constant one-form on that interval. In this way, for a path $\gamma$ in $\Gamma$ the integral $\int_{\gamma}\omega$ is well defined and in this way $\gamma$ defines a linear function on $\Omega (\Gamma)$.
Using cycles on $\Gamma$ this gives rise to a lattice $\Lambda \subset \Omega (\Gamma)^* \cong \mathbb{R}^g$.
Then $J(\Gamma)=\Omega (\Gamma)^*/ \Lambda$.
Fixing a base point $P_0$ in $\Gamma$ one obtains for $d \in \mathbb{Z}$ with $d \geq 1$ an \emph{Abel-Jacobi map} $I(d):\Gamma ^{(d)} \rightarrow J(\Gamma)$.
Here $\Gamma ^{(d)} = \Gamma ^d/ S_d$ is the quotient of $\Gamma ^d$ under the action of the symmetric group using the quotient topology.
This map is defined as follows in \cite{ref4}*{6.2}. For $D=P_1 + \cdots + P_d \in \Gamma ^{(d)}$ choose paths $\gamma _i$ ($1 \leq i \leq d$) in $\Gamma$ from $P_0$ to $P_i$ and consider the element of $\Omega (\Gamma)^*$ defined by $\omega \rightarrow \sum _{i=1}^d \left( \int _{\gamma _i} \omega \right)$.
The image in $J(\Gamma)$ is well-defined (independent from the choices of the paths); this is $I(d)(D)$.
It is proved in \cite{ref4}*{Theorem 6.1} that for $E_1, E_2 \in \Gamma ^{(d)}$ one has that $I(d)(E_1)=I(d)(E_2)$ if and only if $E_1 \sim E_2$. Hence the fibres of $I(d)$ are the linear systems of degree $d$ of $\Gamma$.

On $\Gamma ^{(d)}$ there exists a natural integer affine structure defined and described in \cite{ref9}*{2.1}. This is based on a stratification of $\Gamma ^{(d)}$ into so-called cells of $\Gamma$.
Assume $D \in \Gamma ^{(d)}$. In case $\Supp (D)\subseteq  V(\Gamma)$ then the cell containing $D$ is just $\{ D \}$. Assume $\Supp (D) \nsubseteq V(\Gamma)$ and let $d'$ be the number of points in $\Supp (D) \setminus V(\Gamma)$ and let $e_1, \cdots , e_{d'}$ be the edges containing a particular point of $\Supp (D) \setminus V(\Gamma)$ (those edges need not be mutually different).
Using the identification of $e_i$ with an open interval $e_i\subset \mathbb{R}$, then the cell $\mathfrak{e}$ containing $D$ is identified with an open subset of $\prod_{i=1}^{d'}(e_i) \subset \mathbb{R}^{d'}$ (it is equal to $\prod_{i=1}^{d'}(e_i)$ in case all edges are mutually different, see \cite{ref9} for the other situations).
In particular $\mathfrak{e}$ is identified with an open subset of $\mathbb{R}^{d'}$ and we write $\dim (\mathfrak{e})=d'$.
Clearly, taking $P\in \Supp (D) \setminus V(\Gamma)$ then the subset $\mathfrak{e}_P= \{ D' \in \mathfrak{e} : P \in D' \}$ is an affine hyperplane section of $\mathfrak{e} \subset \mathbb{R}^{d'}$ through $D$.

Let $\mathfrak{e}$ be a cell of $\Gamma ^{(d)}$ of dimension $d'$ and consider the identification $\mathfrak{e} \subset \mathbb{R}^{d'}$.
From the description of $I(d)$ using integrals we obtain that there is an affine map $i_{\mathfrak{e}} : \mathbb{R}^{d'} \rightarrow \mathbb{R}^g$ such that the restriction of $I(d)$ to $\mathfrak{e}$ is equal to the composition of the restriction of the affine map to $\mathfrak{e}$  with the natural quotient map $\mathbb{R}^g \rightarrow J(\Gamma)$.
In case the rank of $i_{\mathfrak{e}}$ is equal to $f \leq d'$ then $i_{\mathfrak{e}}(\mathfrak{e})$ is an open subset of some affine subspace of $\mathbb{R}^g$ of dimension $f$.
In this case we write $\dim (I(d))(\mathfrak{e})=f$ and we have $\dim (I(d)(\mathfrak{e})) \leq \dim (\mathfrak{e})$.

We write $\Gamma _0^{(d)} \subset \Gamma ^{(d)}$ to denote the subset of the divisors $D$ containing some vertex in their support. It is a finite union of cells of $\Gamma ^{(d)}$ of dimension less then $d$. In particular $\dim (I(d) (\Gamma_0^{(d)}))<d$.

From now on we assume $\Gamma$ is a metric graph of genus $g$ such that there is some integer $2 \leq r \leq g-2$ and a linear system $g^r_{2r}$ on $\Gamma$.
Let $T$ be a spanning tree for $\Gamma$ and let $e_1, \cdots, e_g$ be the edges of $\Gamma$ not contained in $T$.
Take $Q_i \in e_i$ for $1 \leq i \leq g$. In case $\Gamma$ is the union of $T$ and $g$ loops then $\Gamma$ has a $g^1_2$.
Indeed, let $P_i$ be the unique point of $\bar{e_i} \setminus e_i$ (hence $P_i \in T$) and consider the divisor $D=2P_1$.
Let $Q\in T$ then using the tree one can make a rational function $f$ on $\Gamma$ (constant on the loops) such that $\divisor (f)=D-2Q$.
In particular $2P_i \in \mid D \mid$ for $1 \leq i \leq g$. Using the loop $e_i$ one can make a rational function $f$ on $\Gamma$ (constant on $\Gamma \setminus e_i$) showing that for  each $Q \in e_i$ there exist $Q' \in e_i$ such that $Q+Q' \in \mid D \mid$ (see Figure \ref{Figuur 3}).
\begin{figure}[h]
\begin{center}
\includegraphics[height=3 cm]{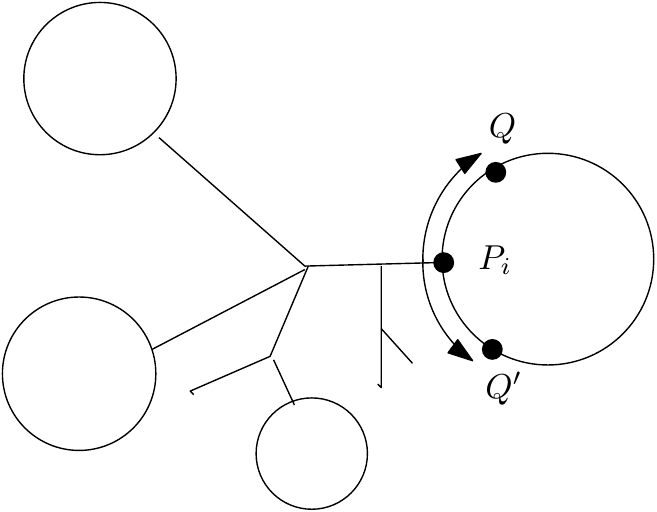}
\caption{Tree with loops }\label{Figuur 3}
\end{center}
\end{figure}
So we can assume $e_g$ is not a loop of $\Gamma$. Then inside $T \cup e_g$ there is another spanning tree $T'$ of $\Gamma$ and an edge $e_{g+1} \notin \{e_1, \cdots , e_g \}$ such that $T' \cup e_{g+1} = T \cup e_g$.
Take $Q_{g+1} \in e_{g+1}$, this is a point of $T$. From \cite{ref6}*{Proposition 3.28} it follows $\{ Q_1, \cdots Q_{g+1} \}$ is a rank-determining set for $\Gamma$.
Therefore it is enough to prove that for each $1 \leq i \leq g+1$ there exists $P_i \in \Gamma$ such that for $1 \leq i < j \leq g+1$ one has $P_i+Q_i \sim P_j+Q_j$.
In the rest of this first part, in order to describe the desired situation, we only use $e_1, \cdots , e_{g-1}$.

Let $D=Q_1 + \cdots Q_{g-1} \in \Gamma ^{(g-1)}$.
For each $P\in \Gamma$ and $S \subseteq \Supp (D) \setminus \{ P \}$ by construction the set $\mathcal{U} _{S,P}$ is equal to $\Gamma \setminus S$, hence $\mathcal{U}^c_{S,P}=S$ and for $Q \in S$ the outgoing degree of $S$ at $Q$ is equal to 2 while $D(Q)=1$.
This implies $D$ is $P$-reduced for all $P\in \Gamma$ and therefore $\mid D \mid = \{ D \}$ (see Lemma \ref{lemma1}). The divisor $D$ belongs to a cell $\mathfrak{e}$ of $\Gamma ^{(g-1)}$ identified with $\prod_{i=1}^{g-1} e_i \subset \mathbb{R}^{g-1}$ of dimension $g-1$ and the restriction $I(g-1) : \mathfrak{e} \rightarrow J(\Gamma)$ is injective.
Hence the affine map $i_{\mathfrak{e}}$ has rank $g-1$ and $\dim \left( I(g-1)(\mathfrak{e}) \right)=g-1$.
In order to make notations not to involved, in such case we are also going to write $\mathfrak{e}$ to denote its image in $J(\Gamma)$.
Let $D_1=Q_1 + \cdots + Q_r$ and $D_2=Q_{r+1} + \cdots + Q_{g-1}$. Also $\mid D_1 \mid = \{ D_1 \}$ and $\mid D_2 \mid = \{ D_2 \}$.
The cell of $\Gamma ^{(r)}$ (resp. $\Gamma ^{(g-1-r)}$) containing $D_1$ (resp. $D_2$) are denoted by $\mathfrak{e}_1$ (resp. $\mathfrak{e}_2$) and the restrictions $I(r): \mathfrak{e}_1 \rightarrow J(\Gamma)$ (resp. $I(g-1-r):\mathfrak{e}_2 \rightarrow J(\Gamma)$) are injective and have images of dimensions $r$ (resp. $g-r-1$), also denoted by $\mathfrak{e}_1$ (resp. $\mathfrak{e}_2$).

Let $k=I(2g-2)(K_{\Gamma})$ and for $F \in g^r_{2r}$ let $h=I(2r)(F)$.
On $J(\Gamma)$ we consider the translated $k- \mathfrak{e}$. This is equal to the projection on $J(\Gamma)$ of an open subset (that we are also going to denote by $k-\mathfrak{e}$) of some affine hyperplane of $\mathbb{R}^g$ (of course it is a translation of $i_{\mathfrak{e}}(\mathfrak{e})$).
For a cell $\mathfrak{c}$ of $\Gamma ^{(g-1)}$, using suited paths on $\Gamma$ to define $i_{\mathfrak{c}}$, the intersection of $I(g-1)(\mathfrak{c})$ and $k-\mathfrak{e}$ is equal to the image of the intersection of $i_{\mathfrak{c}}(\mathfrak{c})$ and $k-\mathfrak{e}$ in $\mathbb{R}^g$.

For $D \in \mathfrak{e}$ one has $r(D)=0$ hence from the Riemann-Roch Theorem it follows $r(K_{\Gamma} - D)=0$.
Hence there exists $E \in \Gamma^{(g-1)}$ such that $I(g-1)(E)=k-I(g-1)(D)$, hence $k-\mathfrak{e} \subset I(g-1)(\Gamma ^{(g-1)})$.
Let $\mathfrak{c}$ be the cell of $\Gamma ^{(g-1)}$ containing $E$.
 Remember that $i_{\mathfrak{c}}(\mathfrak{c})$ is an open subset of some affine subspace $L_{\mathfrak{c}}$ of $\mathbb{R}^g$ of dimension $\dim \left( I(g-1) (\mathfrak{c}) \right)$, hence the intersection of  $i_{\mathfrak{c}}(\mathfrak{c})$ and $k-\mathfrak{e}$  in $\mathbb{R}^g$  is contained in the image in $J(\Gamma)$ of the intersection of $L_{\mathfrak{c}}$ and $k-\mathfrak{e}$.
In case $\mathfrak{c}$ is a cell contained in $\Gamma _0^{(g-1)}$ then $\dim (L_{\mathfrak{c}}) \leq g-2$. Therefore there is an open dense subset of $k-\mathfrak{e} \subset J(\Gamma)$ that is not contained in $ I(g-1) (\Gamma _0^{(g-1)})$.
From now on we assume $Q_1, \cdots, Q_{g-1}$ are chosen such that $k-I(g-1)(D) \notin I(g-1)(\Gamma _0^{(g-1)})$.

Let $E \in \Gamma ^{(g-1)}$ such that $I(g-1)(E) = k-I(g-1)(D)$, hence $D+E \in \mid K _{\Gamma} \mid $.
From Lemma \ref{lemma2} it follows that if there exists some point $P \in \Gamma$ such that $E$ is not $P$-reduced then $E$ is linearly equivalent to some $E' \in \Gamma _0^{(g-1)}$.
Because of our choice of $D$ this is not possible, hence it follows $E$ is $P$-reduced for all $P \in \Gamma$. From Lemma \ref{lemma1} this implies $\mid E \mid = \{ E \} $.
Write $E = P_1 + \cdots P_{g-1}$ and for $1 \leq i \leq g-1$ write $f_i \in E(\Gamma)$ with $P_i \in f_i$. For $1 \leq i < j \leq g-1$ one has $f_i \neq f_j$ otherwise one has $P_i+P_j \sim P'_i + P'_j$ for some other points on $f_i$, contradicting $ \mid E \mid = \{ E \}$.
Let $\mathfrak{f}$ be the cell of $\Gamma ^{(g-1)}$ containing $E$.
We obtained the restriction $I(g-1):\mathfrak{f} \rightarrow J(\Gamma)$ is injective, hence as in the case of $\mathfrak{e}$, one finds that this restriction comes from an affine map $i_{\mathfrak{f}}$ of rank $g-1$ and $\dim \left( I(g-1)(\mathfrak{f}) \right) = g-1$. Again we also write $\mathfrak{f}$ to denote $I(g-1)(\mathfrak{f})$.
The intersection of $\mathfrak{f}$ and $k-\mathfrak{e}$ on $J(\Gamma)$ can be considered as the projection on $J(\Gamma)$ of the intersection of $i_{\mathfrak{f}}(\mathfrak{f})$ and $k-\mathfrak{e}$ inside $\mathbb{R}^g$. This is contained inside the intersection of $k-\mathfrak{e}$ and the affine hyperplane $L_{\mathfrak{f}}$ of $\mathbb{R}^g$ containing $i_{\mathfrak{f}}(\mathfrak{f})$.
So we obtain that a dense open subset of $k-\mathfrak{e}$ in $\mathbb{R}^g$ is covered by finitely many such intersections. Hence there exist such $\mathfrak{f}$ such that $k-\mathfrak{e} \subset L_{\mathfrak{f}}$. There is an open dense subset of $\mathfrak{e}$ such that for $D$ in that open subset and $E \in \mid K_{\Gamma} -D \mid$ this situation occurs.
From now on we assume $Q_1, \cdots , Q_{g-1}$ are chosen such that $D$ belongs to that subset of $\mathfrak{e}$.
For that divisor $D$ and the corresponding divisor $E \in \mathfrak{f}$ there exist neighbourhoods $\mathcal{U}$ (resp $\mathcal{V}$) of $D$ (resp. $E$) in $\mathfrak{e}$ (resp. $\mathfrak{f}$) and an affine isomorphism $\mathbb{R}^{g-1} \rightarrow \mathbb{R}^{g-1}$ such that the restriction to $\mathfrak{e} \subset \mathbb{R}^{g-1}$ induces a bijection $\mu : \mathcal{U} \rightarrow \mathcal{V}$ such that for $D' \in \mathcal{U}$ and $\mu (D') = E'$ one has $D'+E' \in \mid K_{\Gamma} \mid$ (and as we already know $\mid D' \mid = \{ D' \}$ and $\mid E' \mid = \{ E' \}$).

By definition there exists $E_1 \in \Gamma ^{(r)}$ such that $D_1+E_1 \in g^r_{2r}$, hence $I(r)(E_1)=h-I(r)(D_1)\in h-\mathfrak{e}_1$.
Also by the Riemann-Roch Theorem $\mid K_{\Gamma}-g^r_{2r} \mid = g^{g-r-1}_{2g-2r-2}$ hence there exists $E_2 \in \Gamma ^{(g-r-1)}$ such that $D_2+E_2 \in g^{g-r-1}_{2g-2r-2}$, hence $I(g-r-1)(E_2)=k-h-I(g-r-1)(D_2)\in k-h-\mathfrak{e}_2$.
But then $D+E_1+E_2 \in \mid K_{\Gamma} \mid$, hence $E_1+E_2=E$.
Assume the numbering of the points is such that $P_1 + \cdots P_r = E_1$ and $P_{r+1} + \cdots + P_{g-1} = E_2$.
Let $\mathfrak{f}_1$ (resp. $\mathfrak{f}_2$)  be the cell of $\Gamma ^{(r)}$ (resp. $\Gamma ^{(g-r-1)}$) containing $E_1$ (resp. $E_2$).
Repeating the arguments for obtaining the bijection $\mu$ (and if necessary changing the choice of $Q_1, \cdots, Q_{g-1}$ such that $D$ belongs to a suited dense open subset of $\mathfrak{e}$), we obtain for $i=1,2$ the existence of neighbourhoods $\mathcal{U}_i$ (resp. $\mathcal{V}_i$) of $D_i$ (resp $E_i$) in $\mathfrak{e}_i$ (resp. $\mathfrak{f}_i$) and affine isomorphisms $\mathbb{R}^r \rightarrow \mathbb{R}^r$ in case $i=1$ and $\mathbb{R}^{g-r-1} \rightarrow \mathbb{R}^{g-r-1}$ in case $i=2$ inducing bijections $\mu _i : \mathcal{U}_i \rightarrow \mathcal{V}_i$ such that for $D'_i \in \mathcal{U}_i$ and $E'_i=\mu _i(D'_i)$ one has $D'_1+E'_1 \in g^r_{2r}$ or $D'_2+E'_2 \in \mid K_{\Gamma}-g^r_{2r} \mid$.
In particular $D'_1+D'_2+E'_1+E'_2 \in \mid K_{\Gamma} \mid$ hence $E'_1+E'_2=\mu (D'_1+D'_2)$.
Hence by restriction we can assume $\mathcal{U}=\mathcal{U}_1 \times \mathcal{U}_2$, $\mathcal{V}=\mathcal{V}_1\times \mathcal{V}_2$ and $\mu = \mu _1 \times \mu _2$.

\section{Second part  of the proof}\label{section4}

We continue to use $D=D_1+D_2$ and $E=E_1+E_2$ as obtained in Section \ref{section3}.
We are going to prove that there is some suited numbering of the points for $E_1$ and for $E_2$ such that for all $1\leq i < j \leq g-1$ one has $P_i+Q_i \sim P_j+Q_j$.
Let $E''_1=E_1-P_1$, $E'_1=E''_1+P_{r+1}$ (now we are starting to use $r\leq g-2$), $E''_2=E_2-P_{r+1}$ and $E'_2=E''_2+P_1$.
Hence $E'_1+E'_2=E$ is another decomposition of $E$.
There exists $D'_1 \in \Gamma ^{(r)}$ such that $E'_1+D'_1 \in g^r_{2r}$ and there exists $D'_2 \in \Gamma ^{(g-r-1)}$ such that $E'_2+D'_2 \in \mid K_{\Gamma}-g^r_{2r} \mid$.
It follows $E+D'_1+D'_2 \in \mid K_{\Gamma} \mid$ hence $D'_1+D'_2=D$ is also a decomposition of $D$.
One has
\[
E''_1+D=(E_1+D_1-P_1)+D_2 \in \mid g^r_{2r}(-P_1) \mid + D_2
\]
and
\[
E''_1+D=(E'_1+D'_1-P_{r+1})+D'_2 \in \mid g^r_{2r}(-P_{r+1}) \mid + D'_2 \text{ .}
\]
From now on we have to use different arguments than in the proof for curves given in the introduction. Let $1 \leq i \leq r+1$ be such that $\deg (D_1 \cap D'_1)=r-i+1$. We assume the numbering of the points $Q_i$ such that $D''_1=D_1 \cap D'_1=Q_i + \cdots Q_r$ (of course in case $i$ would be $r+1$ then $D''_1=0$).
In case $i=1$ we have $D_1=D'_1$ hence also $D_2=D'_2$ and we obtain $\mid g^r_{2r}(-P_1) \mid = \mid g^r_{2r}(-P_{r+1}) \mid $ hence $P_1 \sim P_{r+1}$. Since $P_1 \neq P_{r+1}$ and $\mid E \mid = \{ E \}$ this is not possible, hence $i \geq 2$.

We are going to prove that $i=2$ hence $D_1=D''_1+Q_1$ and $D'_1=D''_1+Q$ for some $Q \in D_2$ (we can assume the numbering is such that $Q=Q_{r+1}$) and we find
\[
E_1+D_1=E''_1+D''_1+P_1+Q_1 \in g^r_{2r}
\]
and
\[
E'_1+D'_1=E''_1+D''_1+P_{r+1}+Q_{r+1} \in g^r_{2r}
\]
hence $P_1+Q_1 \sim P_{r+1}+Q_{r+1}$.
Replacing $P_{r+1}$ by $P_{r+2}$ (in case $r<g-2$) (resp. $P_1$ by $P_2$) we obtain a similar conclusion but it is important that we use the same $Q_1$ (resp. $Q_{r+1}$) in that conclusion.
From the argument that $i=2$ we are going to obtain that, starting with $Q_1$ (resp. $Q_{r+1}$) there is a characterisation of $P_1$ (resp. $P_{r+1}$)  only using the bijection $\mu_1: \mathcal{U}_1 \rightarrow \mathcal{V}_1$ (resp. $\mu_2 : \mathcal{U}_2 \rightarrow \mathcal{V}_2$).
This characterisation settles that problem.

So we already know that $i \geq 2$. Write $D_1=D''_1+\mathcal{D}_1$ and $D'_1=D''_1+\mathcal{D}'_1$, hence $\mathcal{D}_1 \leq D'_2$ and $\mathcal{D}'_1\leq D_2$ with $\mathcal{D}_1 , \mathcal{D}'_1 \in \Gamma ^{(i-1)}$.
Since 
\[
E''_1+D-\mathcal{D}_1 \in \mid g^r_{2r}(-P_{r+1}) \mid + (D'_2-\mathcal{D}_1)
\]
with $D'_2-\mathcal{D}_1\geq 0$ it follows
\[
r(E''_1+D-\mathcal{D}_1)=r(E''_1+D''_1+D_2)\geq r-1 \text { .}
\]
The divisor $E''_1$ belongs to the cell $\mathfrak{f}'_1$ of $\Gamma ^{(r-1)}$ identified with $\prod_{i=2}^r f_i \subset \mathbb{R}^{r-1}$.
Summation gives rise to $P_1 + \mathfrak{f}'_1\subset \mathfrak{f}_1$, it is the intersection of $\mathfrak{f}_1 \subset \mathbb{R}^r$ with an affine hyperplane.
Let $\mathcal{V}'_1$ be a neighbourhood of $E''_1$ in $\mathfrak{f}'_1$ such that $P_1+\mathcal{V}'_1 \subset \mathcal{V}_1$.
Then we obtain $(\mu_1)^{-1}(P_1+\mathcal{V}'_1)$ is an open subset of the intersection of $\mathfrak{e}_1 \subset \mathbb{R}^r$ with an affine hyperplane.
Let $F''_1 \in \mathcal{V}'_1$ and let $G_1=(\mu _1)^{-1}(P_1+F''_1)$, hence $P_1+F''_1+G_1 \in g^r_{2r}$ and $P_1+F''_1+E_2+G_1+D_2 \in \mid K_{\Gamma} \mid$.
Since $P_1+F''_1+E_2 \in \mathcal{V}$ it follows $\mu (P_1+F''_1+E_2)=G_1+D_2$.

There exists $G_2 \in \Gamma ^{(g-i)}$ such that
\[
F''_1+G_2 \in \mid E''_1+D''_1+D_2 \mid = \mid K_{\Gamma} -P_1-E_2-\mathcal{D}_1 \mid
\]
hence
\[
F''_1+G_2+P_1+E_2+\mathcal{D}_1 \in \mid K_{\Gamma} \mid
\]
hence $\mu (P_1+F''_1+G_2)=G_2+\mathcal{D}_1$.
This implies $G_2+\mathcal{D}_1 = G_1+D_2$.
Since $\mathcal{D}_1 \cap D_2 = \emptyset$ it follows $D_2 \leq G_2$ hence $G_2=D_2+G'_2$ for some $G'_2 \in \Gamma ^{(r-i+1)}$.
This implies $G_1=G'_2+\mathcal{D}_1$. So we obtain
\[
(\mu _1)^{-1}(P_1+\mathcal{V}'_1) \subseteq \mathcal{D}_1 + \Gamma ^{(r-i+1)} \text { .}
\]
Since $\mu _1$ is a bijection and $\dim (P_1+\mathcal {V}'_1)=r-1$ it follows $r-1 \leq r-i+1$ hence $i \leq 2$ and therefore $i=2$.

In this case of equality the argument also implies
\[
(\mu _1)^{-1}(P_1+\mathcal{V}'_1) \subseteq Q_1 + \Gamma ^{(r-1)}
\]
Since $E_1 \in P_1+\mathcal{V}'_1$ it follows $D_1=Q_1+D''_1 \in (\mu _1)^{-1} (P_1+\mathcal{V}'_1)$.
The divisor $D''_1$ belongs to the cell $\mathfrak{e}'_1$ of $\Gamma ^{(r-1)}$ identified with $\prod_{i=2}^r e_i \subset \mathbb{R}^{r-1}$.
Since $(\mu _1)^{-1}(P_1+\mathcal{V}'_1)$  is an open subset of an affine hyperplane section of $\mathfrak{e}_1 \subset \mathbb{R}^r$  and  $Q_1+\mathfrak{e}'_1$  is an affine hyperplane section of $\mathfrak{e}_1 \subset \mathbb{R}^r$, because of the inclusion it is the same hyperplane section.
Therefore there exists a neighbourhood $\mathcal{U}'_1$ of $D''_1$ in $\mathfrak{e}'_1$ such that $\mu _1$ induces a bijection between $P_1+\mathcal{V}'_1$ and $Q_1+\mathcal{U}'_1$.
This implies $P_1$ can be characterised by the following property.
It is the unique point $P$ on $\Gamma$ such that for all $D'$ contained in a sufficiently small neighbourhood of $D_1$ in $\mathfrak{e}_1$ with $Q_1 \in D'$ the divisor $E' = \mu _1 (D')$ contains $P$.
This implies that, taking any $P_j$ (with $r+2 \leq j \leq g-1$) instead of $P_{r+1}$ (in case $r<g-2$), there exist $Q \in D_2$ such that $P_1+Q_1 \sim P_j+Q$ (because the point $P_1$ is characterised by $\mu _1$ and $Q_1$).
Also taking any $P_i$ (with $2 \leq i\leq g-1$) instead of $P_1$ there exist $Q \in D_1$ and $Q' \in D_2$ such that $P_i+Q \sim P_{r+1} + Q'$.
We need to conclude that $Q'=Q_{r+1}$. In case $r=g-2$ this trivially holds, so we assume $r<g-2$.

The divisor $E''_2$ belongs to the cell $\mathfrak{f}'_2$ of $\Gamma ^{(g-r-2)}$ identified with $\prod_{i=r+2}^g f_i \subset \mathbb{R}^{g-r-2}$ and by summation we have $P_{r+1}+\mathfrak{f}'_2 \subseteq \mathfrak{f}_2$.
Let $\mathcal{V}'_2$ be a neighbourhood of $E''_2$ in $\mathfrak{f}'_2$ such that $P_{r+1}+\mathcal{V}'_2 \subseteq \mathcal{V}_2$.
Let $F_2 \in \mathcal{V}'_2$ and let $G_2=(\mu _2)^{-1}(P_{r+1}+F_2)$, hence $P_{r+1}+F_2+G_2 \in \mid K_{\Gamma} - g^r_{2r} \mid$.
Using $D_1+G_2$ instead of $D_1+D_2$ we obtain there exists $Q'_{r+1} \in G_2$ such that $P_1+Q_1 \sim P_{r+1}+Q'_{r+1}$ (here we use again that $P_1$ is characterized by $\mu _1$ and $Q_1$).
Hence also $P_{r+1} + Q_{r+1} \sim P_{r+1} + Q'_{r+1}$ hence $Q_{r+1} \sim Q'_{r+1}$. Since $\mid D \mid = \{ D \}$ it follows $Q_{r+1}=Q'_{r+1}$.
This implies $G_2 \in Q_{r+1} + \Gamma ^{(g-r-2)}$ hence
\[
(\mu _2)^{-1}(P_{r+1}+\mathcal{V}'_2) \subset Q_{r+1} + \Gamma ^{(g-r-1)} \text { .}
\]
One has $(\mu _2)^{-1}(P_{r+1}+E''_2)=(\mu _2)^{-1}(E_2)=D_2=Q_{r+1}+D''_2$ with $D''_2=D_2-Q_{r+1}$. Let $\mathfrak{e}'_2$ be the cell of $\Gamma ^{(g-2-r)}$ containing $D''_2$. Again $Q_{r+1} + \mathfrak{e}'_2$  is an affine hyperplane section in $\mathfrak{e}_2 \subset \mathbb{R}^{g-r-1}$.  One concludes as before that there exists a neighbourhood $\mathcal{U}'_2$ of $D''_2$ in $\mathfrak{e}'_2$ such that 
\[
\mu _2 (Q_{r+1} + \mathcal {U}'_2)=P_{r+1}+\mathcal{V}'_2 \text { .}
\]
We obtain the following characterisation for $P_{r+1}$ using $\mu _2$ and $Q_{r+1}$. 
It is the point $P \in \Gamma$ such that for all $D''$ in a sufficiently small neighbourhood of $D_2$ in $\mathfrak{e}_2$ containing $Q_{r+1}$ the divisor $E''=\mu _2(D'')$ contains $P$.
From this characterisation it follows that for each $2 \leq i \leq r$ there exists $Q \in D_1$ such that $P_i+Q \sim P_{r+1}+Q_{r+1}$.
In case $Q_1=Q$ it would follow that $P_1+Q_1 \sim P_i+Q_1$ hence $P_1 \sim P_i$. Since $\mid E \mid = \{ E \}$ this is impossible.
In case $r<g-2$ taking $r+2 \leq j \leq g-1$ and taking $P_j$ instead of $P_{r+1}$ we obtain $Q \in D_2$ such that $P_1+Q_1 \sim P_j+Q$.
If $Q=Q_{r+1}$ we obtain $P_{r+1} + Q_{r+1} \sim P_j + Q_{r+1}$ hence $P_{r+1} \sim P_j$, again a contradiction.
This proves we can use a numbering $E=P_1 + \cdots + P_{g-1}$ such that for $1 \leq i < j \leq g-1$ one has $P_i+Q_i \sim P_j+Q_j$.

Now taking $Q_1, \cdots , Q_{g-2}, Q_g$ (resp. $Q_1, \cdots , Q_{g-2}, Q_{g+1}$) instead of $Q_1, \cdots , Q_{g-1}$ we still use the same bijection $\mu _1$ (once again we use $r \leq g-2$) and we find $P_g \in \Gamma$ (resp. $P_{g+1} \in \Gamma$) such that $P_1+Q_1 \sim P_g+Q_g$ (resp. $P_1+Q_1 \sim P_{g+1}+Q_{g+1}$).
Since $P_1, \cdots , P_{g+1}$ is a rank-determining set this proves $r(P_1+Q_1)= 1$, hence $\Gamma$ has a linear system $g^1_2$.

\end{document}